\documentclass{amsart}

\usepackage{amsmath, amssymb, amsthm}
\usepackage[mathcal]{eucal}
\usepackage{enumerate}

\newtheorem{theorem}{Theorem}

\newtheorem{lemma}{Lemma}[section]

\numberwithin{equation}{section}

\def\ZA{\ensuremath{\mathcal A}}

\def\ZS{\ensuremath{\mathcal S}}

\def\ZR{\ensuremath{\mathbb R}}

\def\Z1{\ensuremath{\mathbf 1}}

\newcommand {\e }[1]{(\ref{#1})}
\newcommand {\lem }[1]{Lemma \ref{#1}}

\newcommand {\trm }[1]{Theorem \ref{#1}}

\newcommand {\sect }[1]{Section \ref{#1}}

\newcommand\numberthis{\addtocounter{equation}{1}\tag{\theequation}}

\author{Gevorg Mnatsakanyan}
\title[Weighted strong-sparse estimates]{Sharp weighted estimates for strong-sparse operators}

\address{Mathematical Institute, University of Bonn, Endenicher Allee 60, 53115 Bonn, Germany}
\curraddr{}
\email{gevorg@math.uni-bonn.de}


\usepackage[dvipsinames]{xcolor}
\definecolor{midnightblue}{HTML}{0059b3}
\definecolor{chromered}{HTML}{f14233}
 \RequirePackage[colorlinks,citecolor=midnightblue,urlcolor=midnightblue,breaklinks=true,
linkbordercolor=midnightblue,linkcolor=midnightblue]{hyperref}

\begin{document}

\maketitle

\begin{abstract}
    We prove the sharp weighted-$L^2$ bounds for the strong-sparse operators introduced in \cite{KaragulyanM}. The main contribution of the paper is the construction of a weight that is a lacunary mixture of dual power weights. This weights helps to prove the sharpness of the trivial upper bound of the operator norm.
\end{abstract}

\section{Introduction}
The theory of weighted inequalities started with the seminal work of Muckenhoupt \cite{Muck}, where he proved that the Hardy-Littlewood maximal operator is bounded on $L^p(w)$, $1< p< \infty$, for positive measurable $w :\ZR \to \ZR$ if and only if
\begin{equation}\label{Muckenhouptclass}
    [w]_{A_p} := \sup\limits_{I}  \left( \frac{1}{|I|} \int_I w \right) \left( \frac{1}{|I|} \int_I w^{-\frac{1}{p-1}} \right)^{p-1} < \infty,
\end{equation}
where the supremum is taken over all intervals and $|I|$ denotes the Lebesgue measure of the interval. If \e{Muckenhouptclass} holds, then $w$ is said to be in the Muckenhoupt class $A_p$ and the quantity $[w]_{A_p}$ is called its $A_p$ characteristic. Later, Buckley \cite{Buck} obtained the sharp dependence of the norm of the maximal operator on the $A_p$ characteristic. Namely, he proved that
\begin{align}
    \| M \|_{L^p(w) \to L^{p,\infty} (w)} \lesssim [w]_{A_p}^{ \frac{1}{p} }, \label{sharpweightedweakmaximal}\\
    \| M \|_{L^p(w) \to L^{p} (w)} \lesssim [w]_{A_p}^{ \frac{1}{p-1} }, \label{sharpweightedmaximal}
\end{align}
and these are sharp in the sense of the theorems below.

The problem of the sharp dependence of the $L^2(w) \to L^2(w)$ norm of the Cald\'eron-Zygmund operator on the $A_2$ characteristic of $w$ is known as the $A_2$-conjecture. It was first proved by Hyt\"onen \cite{Hyt2, Hyt}. A simpler proof was given by Lerner \cite{Ler, Ler2} proving that the Cald\'eron-Zygmund operators can be dominated by the simple sparse operators. Later, it was proved that a number of operators in harmonic analysis admit pointwise or norm domination by the sparse operators \cite{CondeAlonsoRey, Ler3, Kar1, Lac2, Ler, Ler2}. On the other hand, $L^p$ and weighted-$L^p$ bounds for the sparse operators are fairly easy to obtain\cite{Lac2}.

Let us have a family $\ZS$ of intervals in $\ZR$ and $0 < \gamma <1$. $\ZS$ is called $\gamma$-sparse, or just sparse, if there exists pairwise disjoint subsets $E_A \subset A$, $A\in \ZS$, such that $|E_A| \geq \gamma |A|$.
Let us set for an interval $B$
\begin{equation*}
    \langle f \rangle_B := \frac{1}{|B|} \int_B |f|, \quad M_B f:= \sup\limits_{ A \text{ intervals} :\: A\supset B} \langle f \rangle_A.
\end{equation*}

For a sparse family $\ZS$, we define the sparse and the strong-sparse operators as
\begin{align}
    \ZA _{\ZS} f (x) := \sum\limits_{A \in \ZS} \langle f \rangle_A \cdot \Z1_A (x), \\
    \ZA^* _{\ZS} f (x) := \sum\limits_{A \in \ZS} ( M_A f ) \cdot \Z1_A (x),
\end{align}
respectively.
The sharp weighted bound for the sparse operator\cite{Lac2} is as follows
\begin{equation}\label{sharpweightedsparse}
    \| \ZA_{\ZS} \|_{L^p(w) \to L^p(w) } \lesssim [w]_{A_p}^{ \max (1 , \frac{1}{p-1} ) }.
\end{equation}

The strong-sparse operators were introduced by Karagulyan and the author in \cite{KaragulyanM}, where $L^p$ and weak-$L^1$ estimates are proved in the setting of an abstract measure space with ball-basis. In this paper, we obtain the sharp dependence of the weighted-$L^2$ norm of the strong-sparse operator on the $A_2$ characteristic of the weight.
\begin{theorem}\label{Weak}
For an $A_2$ weight $w$ we have the bound
\begin{equation}
    \| \ZA^*_{\ZS} \|_{L^2 (w) \to L^{2,\infty} (w) } \lesssim [w]_{A_2}^{\frac{3}{2} }.
\end{equation}
The inequality is sharp in the following sense: there exist a sparse family $\ZS$ and a sequence of weights $w_{\alpha}$ such that
\begin{equation}
    [w_{\alpha}]_{A_2} \to \infty, \text{ as } \alpha \to 0,
\end{equation}
and for any function $\phi : [0, \infty ) \to [0, \infty )$ with $\phi (x) / x^{\frac{3}{2}} \to 0$ as $x \to \infty$, we have
\begin{equation}
    \frac{ \| \ZA^*_{\ZS} \|_{L^2 (w_{\alpha}) \to L^{2,\infty} (w_{\alpha}) } }{ \phi ( [w_{\alpha}]_{A_2} ) }  \to \infty, \text{ as } \alpha \to 0. 
\end{equation}
\end{theorem}
\begin{theorem}\label{Strong}
For an $A_2$ weight $w$ we have the bound
\begin{equation}
    \| \ZA^*_{\ZS} \|_{L^2 (w) \to L^2(w) } \lesssim [w]_{A_2}^2.
\end{equation}
The inequality is sharp in the following sense: there exist a sparse family $\ZS$ and a sequence of weights $w_{\alpha}$ such that
\begin{equation}
    [w_{\alpha}]_{A_2} \to \infty, \text{ as } \alpha \to 0,
\end{equation}
and for any function $\phi : [0, \infty ) \to [0, \infty )$ with $\phi (x) / x^2 \to 0$ as $x \to \infty$, we have
\begin{equation}
    \frac{ \| \ZA^*_{\ZS} \|_{L^2 (w_{\alpha}) \to L^2 (w_{\alpha}) } }{ \phi ( [w_{\alpha}]_{A_2} ) }  \to \infty, \text{ as } \alpha \to 0. 
\end{equation}
\end{theorem}
On the other hand, we have the following simple partial improvement for the strong bound.
\begin{theorem}\label{Chain}
Let the sparse family $\ZS$ be such that for any two $A, B \in \ZS$ either $A\subset B$ or $B \subset A$. Then, we have
\begin{equation}
    \| \ZA^*_{\ZS} \|_{L^2 (w) \to L^2(w) } \lesssim [w]_{A_2}^{\frac{3}{2}}.
\end{equation}
\end{theorem}

Looking at the definition of the strong-sparse operators, we see that $M_B f \leq Mf(x)$ for any $x\in B$. Thus, $M_B f \leq \langle Mf \rangle_B$ and we obtain
\begin{equation}
    \ZA^*_{\ZS} f (x) \leq \ZA_{\ZS} (Mf).
\end{equation}
Then, one can try to black-box the sharp weighted bounds \e{sharpweightedweakmaximal}, \e{sharpweightedmaximal} and \e{sharpweightedsparse} for \trm{Weak} and \trm{Strong}. As it will be shown in \sect{WeakBound}, the weighted weak-$L^2$ bound for the sparse operator is the same as for the strong one. Thus, \trm{Weak} will not follow from such a black-box. Instead, we will decompose the operator according to the magnitude of the $M_B f$ for the sparse intervals $B$, then, we will use the weighted weak bound of the maximal function \e{sharpweightedweakmaximal}. We will do this in \sect{WeakBound}.

As for \trm{Strong}, we see that by black-boxing the above mentioned inequalities we trivially get the upper bound, i.e.
\begin{align*}
\| \ZA^*_{\ZS} \|_{L^2 (w) \to L^2 (w) } &\leq \| \ZA_{\ZS} \circ M \|_{L^2 (w) \to L^2 (w) } \\
&\lesssim \| \ZA_{\ZS} \|_{L^2 (w) \to L^2 (w) } \| M \|_{L^2 (w) \to L^2 (w) } \lesssim [w]_{A_2}^{2}.
\end{align*}
Thus, the interesting thing about \trm{Strong} is to obtain the sharpness of this estimate. For that we will construct a weight which is a lacunary mixture of the dual power weights $x^{\alpha - 1}$ and $x^{1-\alpha}$. We will do this in \sect{StrongBound}.

In \sect{ChainCase}, we will prove \trm{Chain}.

We say $a \lesssim b$ if there is an absolute constant $c$, maybe depending on the sparse parameter $\gamma$, such that $a \leq c\cdot b$. Furthermore, we say $a\sim b$ if $a\lesssim b$ and $b \lesssim a$.

\section*{Acknowledgments}
The author is grateful to his advisor Grigori Karagulyan for posing this problem, for introducing him to the field of Harmonic Analysis and for his continuous support.

\section{The upper bound of \trm{Weak}} \label{WeakBound}

\subsection{A well-known property of $A_{\infty}$ weights}
Following \cite{Per2, Per1}, we say that $w$ is an $A_{\infty}$ weights if
\begin{equation}
    [w]_{A_{\infty} } := \sup\limits_{I} \frac{1}{ w(I) } \int_I M( w \Z1_I ) (x) dx < \infty.
\end{equation}
It is well-known that any $A_{p}$ weight is also an $A_{\infty}$ weights and that a reverse H\"older inequality holds for in the latter class. The following theorem with sharp constants is due to Hyt\"onen, P\'erez and Rela\cite{Per1}.
\begin{theorem}
If $w$ is an $A_{\infty}$ weight and $\epsilon = \frac{1}{4 [ w ]_{A_{\infty}} }$, then
\begin{equation*}
\langle w^{1+\epsilon} \rangle _I \leq 2 \left( \langle w \rangle _I \right)^{ 1+\epsilon },
\end{equation*}
for any interval $I$.
\end{theorem}
This implies the following lemma.
\begin{lemma} \label{E<Q}
For any cube $Q$ and measurable subset $E\subset Q$, we have
\begin{equation*}
w(E) \leq 2w(Q) \left( \frac{ |E| }{ |Q| } \right)^{ c / [ w ] _{A_{\infty}} },
\end{equation*}
where $c$ is an absolute constant.
\end{lemma}
\begin{proof}
Let $\epsilon$ be as before.
\begin{align*}
\int_E w & \leq \left( \int_E w^{ 1+\epsilon } \right)^{\frac{1}{1+\epsilon}} \cdot |E|^{  \epsilon / (1+\epsilon ) }  \quad \text{(H\"{o}lder)} \\
& = \langle w^{1+\epsilon} \rangle _Q ^{ \frac{1}{ 1+\epsilon } } \cdot |E|^{  \epsilon / (1+\epsilon ) } \cdot |Q|^{1/(1+\epsilon ) } \\
& \leq 2 \langle w \rangle _Q |E|^{  \epsilon / (1+\epsilon ) } \cdot |Q|^{1/(1+\epsilon ) } \quad \text{(Reverse H\"{o}lder)} \\
& = 2 w(Q) \left( \frac{|E|}{|Q|} \right) ^{ c / [ w ] _{A_{\infty}} }.
\end{align*}
\end{proof}

\subsection{The proof of the weak bound}
The idea is to group $M_Bf$'s, $B \in \ZS$, according to their magnitude and estimate each group applying \lem{E<Q} and the weighted weak bound for the maximal operator \e{sharpweightedweakmaximal}.
Denote $\alpha := \frac{1}{[w]_{\infty}}$, and for $\lambda > 0$ let
\begin{align*}
& A_0 := \{ B\in \ZS \; : \; M_B f > \alpha \lambda \}, \\
& A_j := \{ B\in \ZS \; : \; 2^{-j+1} \alpha \lambda \geq M_B f > 2^{-j} \alpha \lambda \},
\end{align*}
for $j=1,2,\dots$. Thus, $A_j$'s partition $\ZS$. We write
\begin{align*}
w\{ \ZA_{\ZS}^*f > \lambda \} &\leq \sum\limits_{j=0}^{\infty} w \Big\{ \sum\limits_{B\in A_j} ( M_B f ) \chi_B > \lambda 2^{-j/2} C \Big\} \\
&\leq  w \left( \bigcup\limits_{B\in A_0} B \right) + \sum\limits_{j=1}^{\infty} w \Big\{ \sum\limits_{B \in A_j} \chi_B > \frac{1}{\alpha} 2^{j/2} C \Big\} \\
&\leq w \big\{ Mf > \lambda \alpha \big\}  + \sum\limits_{j=1}^{\infty} w \left( \bigcup\limits_{B\in A_j} B \right) \frac{ | \{ \sum\limits_{B \in A_j} \chi_B > \frac{1}{\alpha} 2^{j/2} C \} |^{c/ [w]_{\infty} } }{ | \bigcup\limits_{B\in A_j} B |^{c/ [w]_{\infty} } } \\
&\leq w \big\{ Mf > \lambda \alpha \big\}  + \sum\limits_{j=1}^{\infty} w \left( \bigcup\limits_{B\in A_j} B \right) 2^{ - \frac{c}{\alpha} 2^{j/2} C\alpha } \\
&\leq w \big\{ Mf > \lambda \alpha \big\} + \sum\limits_{j=1}^{\infty} w \big\{ Mf > 2^{-j} \lambda \alpha \big\} 2^{-cC2^{j/2}} \\
&\leq \frac{ [w]_{A_{\infty}}^2 }{\lambda ^2 } \| M \|^2 _{ L^2\rightarrow L^{2,\infty} },
\end{align*}
where the first line is due to the triangle inequality, the third inequality follows from \lem{E<Q} and the fourth one from the fact, that $A_j$ is a sparse collection. It reimains to apply the bound \e{sharpweightedweakmaximal} to get the upper bound of \trm{Weak}.

\subsection{The lower bound of \trm{Weak}}
Let $w=|x|^{\alpha - 1}$ and $\sigma = |x|^{1 - \alpha}$ be the dual power weights, $0< \alpha <1$. We know that
\begin{equation}
    [w]_{A_2} = [\sigma ]_{A_2} \sim \frac{1}{\alpha}.
\end{equation}
Let $\ZS := \{ [0, 2^{-k} ) :\; \text{ for }k \in \mathbb{N} \}$ be a sparse family. Then, we claim
\begin{equation}\label{toprovelowerweak}
    \| \ZA_{\ZS}^* ( \sigma \Z1_{[0 ,1) } ) \|_{ L^{2,\infty} (w) } \sim [w]_{A_2}^{3/2} \| \sigma \Z1_{ [0, 1) } \|_{L^2 (w) }.
\end{equation}
The square of the right-hand side of \e{toprovelowerweak} equals $\frac{1}{(2- \alpha ) \alpha ^3} $. On the other hand,
\begin{align*}
     \| \ZA_{\ZS}^* ( \sigma \Z1_{[0 ,1) } ) \|_{ L^{2,\infty} (w) } ^2 & \geq \frac{1}{\alpha^2 } w \{ \ZA_{\ZS}^* ( \sigma \Z1_{[0 ,1) } ) > \frac{1}{\alpha} \} \\
     & = \frac{1}{\alpha^2 } w \{ \sum\limits_{k=1}^{\infty } \Z1_{ [0, 2^{-k} ) } \gtrsim \frac{1}{\alpha} \} \\
     & = \frac{1}{\alpha^2 } w ( [0, 2^{-\frac{ c }{\alpha } }  ) ) \sim \frac{ 2^{-\frac{c}{\alpha } \alpha } }{\alpha^3}.
\end{align*}
So the proof of \trm{Weak} is complete.

\section{The lower bound of \trm{Strong}} \label{StrongBound}
\subsection{Construction of the weight}
Let $0< \alpha <1$ be small enough integer power of $2$, i.e. $\alpha = 2^{-a}$ for large enough integer $a$. Let us define the weight $\sigma : \ZR \to [0, \infty )$ to be even and
\begin{equation}
    \sigma (x) := \begin{cases}
            \frac{ 2^{2k(1-\alpha)} }{ \alpha } ( x - 2^{- (k + 1) } )^{1-\alpha} , \quad x\in [ 2^{-(k+1)}, (1 + \alpha)2^{ -(k + 1) } ) \text{ for } k\in \mathbb{N} \\
            x^{\alpha - 1},\quad x\in [ (1 + \alpha)2^{-( k + 1) } ) , (1-\alpha)2^{-k} ) \text{ for } k\in \mathbb{N} \\
            \frac{ 2^{2k(1-\alpha)} }{ \alpha } ( 2^{-k} - x )^{1-\alpha}, \quad x \in [(1-\alpha)2^{-k}, 2^{-k} ) \text{ for } k\in \mathbb{N} \\
            x^{\alpha - 1},\quad x\in [ \frac{1}{2}, \infty ).
    \end{cases}
\end{equation}
The dual weight to $\sigma$ is $w(x) := \sigma(x)^{-1}$. We will prove that
\begin{equation}\label{A_2constant}
    \sup_{I} \frac{1}{|I|^2} \big( \int_I w \big) \cdot \big( \int_I \sigma \big) \sim \frac{1}{\alpha},
\end{equation}
that is, $\sigma \in A_2$ with $[\sigma]_{A_2} \sim \frac{1}{\alpha}$.

First, we show that \e{A_2constant} holds for dyadic intervals. Let us partition all dyadic intervals into three groups.
\begin{enumerate}
    \item[a.] $I = [0, 2^{-k})$ for some $k\in \mathbb{N}_0$. Then, we compute
    \begin{align*}
        \int\limits_{2^{ -(k+1) } }^{2^{-k} } w(x)dx &= \int\limits_{(1 + \alpha) 2^{-(k+1)} }^{(1-\alpha )2^{-k} } x^{ 1 - \alpha} dx + \alpha 2^{2k(\alpha - 1) } \int\limits_{(1-\alpha )2^{-k} }^{ 2^{-k} }  ( 2^{-k} - x )^{\alpha - 1} dx  \\
        &\quad + \alpha 2^{2k(\alpha - 1) } \int\limits_{2^{-(k +1) } }^{ (1+ \alpha ) 2^{-(k+1) } }  ( x - 2^{-(k +1) } )^{\alpha - 1} dx \\
        &= \frac{ (1 - \alpha )^{2- \alpha} 2^{ -( 2-\alpha)k } - (1+\alpha)^{2-\alpha } 2^{ -(2 - \alpha)(k+1) } }{ 2- \alpha} \\
        &\quad + \alpha 2^{ 2k(\alpha - 1) } \cdot \frac{ \alpha^{\alpha} ( 2^{-k\alpha} + 2^{-(k+1) \alpha} ) }{\alpha} \\
        &= c(\alpha) 2^{ -k ( 2-\alpha ) }.\numberthis \label{calculationforw}
    \end{align*}
    In the above computations and below $c(\alpha)$ is a constant depending on $\alpha$ absolutely bounded and away from $0$. It will be different at each occurence. Next, we have
    \begin{equation}
        \int\limits_0^{2^{-k} } w(x)dx = \sum\limits_{j=k}^{\infty} \int\limits_{ 2^{-(j+1)} }^{ 2^{-j} } w (x) dx = \sum\limits_{j=k}^{\infty} c(\alpha) 2^{ -j ( 2-\alpha ) } = c(\alpha ) \cdot 2^{ -k ( 2 - \alpha ) }. 
    \end{equation}
    
    For $\sigma$ we have
    \begin{align*}
        \int\limits_{ 2^{ -(k+1) } }^{ 2^{-k} } \sigma (x) dx &= \int\limits_{ (1+\alpha) 2^{ -(k+1) } }^{ (1-\alpha) 2^{ -k} } \sigma (x) dx + \int\limits_{ (1 - \alpha) 2^{-k} }^{ 2^{ -k } } \sigma (x) dx + \int\limits_{ 2^{-(k + 1) } }^{ (1 + \alpha) 2^{ -(k+1) } } \sigma (x) dx \\
        &= \frac{ ( 1 - \alpha )^{ \alpha } 2^{ -k \alpha } - (1 + \alpha)^{\alpha} 2^{ -(k+1 ) \alpha } }{ \alpha } + \frac{ 2^{ 2k (1-\alpha ) } }{\alpha} \cdot \\
        & \quad \cdot \left( \int\limits_{ 2^{-k} ( 1 - \alpha ) }^{ 2^{-k} } ( 2^{-k} - x )^{ 1 - \alpha } dx + \int\limits_{ 2^{-(k+1) } }^{ ( 1 + \alpha ) 2^{-(k + 1) } } ( x - 2^{-(k + 1)} )^{ 1 - \alpha } dx \right) \\
        &= c( \alpha ) 2^{- k \alpha } + \frac{ 2^{2k(1-\alpha )} }{ \alpha } \cdot \frac{ \alpha^{ 2 - \alpha } ( 2^{ -k (2 - \alpha) } + 2^{-(k+1)(2-\alpha ) } ) }{2 - \alpha } \\
        &= c(\alpha) \frac{ 2^{ -k\alpha } }{\alpha} + \alpha 2^{-k\alpha} = c(\alpha)\frac{ 2^{ -k\alpha } }{\alpha}. \numberthis 
    \end{align*}
    Then, we have
    \begin{equation}\label{calculationforsigma}
        \int\limits_0^{ 2^{-k} } \sigma (x) dx = \sum\limits_{j = k}^{ \infty } \int\limits_{ 2^{ -(j+1) } }^{ 2^{ -j} } \sigma (x) dx = \sum\limits_{j = k}^{ \infty } c(\alpha) \frac{ 2^{ -j\alpha } }{\alpha} = c(\alpha)\frac{ 2^{ -k\alpha } }{\alpha}.
    \end{equation}
    Combining the two computations above, we have for \e{A_2constant}
    \begin{equation}
        2^{2k} \big( \int\limits_0^{2^{-k} } w \big) \cdot \big( \int\limits_0^{2^{-k} } \sigma \big) = c(\alpha) 2^{2k} 2^{ - k(2 - \alpha ) } \frac{ 2^{-k \alpha } }{ \alpha} \sim \frac{1}{\alpha}.
    \end{equation}
    \item[b.] One of the following holds: for some $k\in \mathbb{N}_0$ $I\subset [2^{ -(k+1) } , (1+\alpha) 2^{-(k+1)} )$, $I\subset [(1+\alpha)2^{ -(k+1) } , (1-\alpha) 2^{-k} )$ or $I\subset [ (1-\alpha ) 2^{-k}$. On these intervals, the weights $w$ and $\sigma$ are just rescaled versions of the power weights. Thus, we immediately have 
    \begin{equation}\label{caseb}
        \frac{1}{|I|^2} \big( \int_I w \big) \cdot \big( \int_I \sigma \big) \lesssim \frac{1}{\alpha},
    \end{equation}
    by the $A_2$ characteristic of the power weights.
    \item[c.] $I \subset [2^{-(k+1)} , 2^{-k} )$ and either $[ (1-\alpha ) 2^{-k}, 2^{-k} ) \subset I$ or $[ 2^{-(k+1) }, (1 + \alpha ) 2^{-(k + 1) } ) \subset I$ for some $k\in \mathbb{N}_0$. This is the intermediate case between the above two. The computation for the choice of the last two conditions is identical, so we consider only one of them. Let $|I| = 2^{ - m}$ so that $I = [2^{-k} - 2^{-m} , 2^{-k})$ and $k + 2\leq m \leq k+a$, where we recall $\alpha = 2^{-a}$. We start calculating
    \begin{align*}
        \int\limits_{2^{-k} - 2^{-m} }^{2^{-k} } w(x)dx &= \int\limits_{ 2^{-k} - 2^{-m} }^{ (1-\alpha )2^{-k } } x^{ 1 - \alpha} dx + \int\limits_{(1-\alpha )2^{-k} }^{ 2^{-k} }  \Big( 2^{2k} \cdot \frac{2^{-k} - x }{ \alpha } \Big)^{\alpha - 1} dx  \\
        &= \frac{ (1 - 2^{-a} )^{2- \alpha} 2^{ -( 2-\alpha)k } - 2^{ -(2 - \alpha)k } (1-2^{k-m} )^{2- \alpha} }{ 2- \alpha} \\
        &\quad + \alpha^{1-\alpha } \cdot 2^{ 2k(\alpha - 1) } \int\limits_{ 2^{ -k } (1-\alpha) }^{ 2^{-k} } ( 2^{-k} - x )^{\alpha - 1} dx \\
        &= c(\alpha) 2^{ -k ( 2-\alpha ) }
        \Big( ( 1 + \frac{2^{k-m} - 2^{-a} }{ 1 - 2^{k-m} } )^{2-\alpha} - 1 \Big)  +2^{ -k ( 2 - \alpha ) } \\
        &= c(\alpha ) 2^{ - k ( 2 - \alpha ) } 2^{ k - m} + 2^{ - k(2 - \alpha) } = c(\alpha) 2^{-k(2-\alpha)}.
    \end{align*}
    For $\sigma$ we write
    \begin{align*}
        \int\limits_{ 2^{-k} - 2^{-m} }^{ 2^{-k} } \sigma (x) dx &= \int\limits_{ 2^{-k} - 2^{-m} }^{ (1-\alpha) 2^{ -k} } \sigma (x) dx + \int\limits_{ (1 - \alpha) 2^{-k} }^{ 2^{ -k } } \sigma (x) dx \\
        &= \frac{ ( 1 - \alpha )^{ \alpha } 2^{ -k \alpha } - 2^{ -k \alpha } (1 - 2^{k-m} )^{\alpha} }{ \alpha } \\
        &\quad + 2^{ 2k (1-\alpha ) } \cdot \alpha^{ \alpha - 1 } \cdot \int\limits_{ 2^{-k} ( 1 - \alpha ) }^{ 2^{-k} } ( 2^{-k} - x )^{ 1 - \alpha } dx \\
        &= c( \alpha ) 2^{ - k \alpha } \frac{ ( 1 + \frac{2^{k-m} - 2^{-a} }{ 1 - 2^{k-m} } )^{\alpha} - 1 }{\alpha} + \alpha 2^{-k\alpha} \\
        &= c(\alpha) 2^{-k\alpha} \cdot 2^{k-m} + 2^{-a} \cdot 2^{-k \alpha} = c( \alpha ) 2^{-k \alpha + k -m}.
    \end{align*}
    Thus, for \e{A_2constant} we have
    \begin{align*}
        \frac{1}{|I|^2} \big( \int_I w \big) \cdot \big( \int_I \sigma \big) &= 2^{2m} \cdot c(\alpha ) 2^{-k\alpha + k - m} \cdot 2^{ - k (2-\alpha ) } \\
        &\sim 2^{m-k} \lesssim 2^{a} = \frac{1}{\alpha}.
    \end{align*}
\end{enumerate}

We conclude, that the dyadic $A_2$ characteristic of $w$ is $\frac{c}{\alpha}$. It is important here, that the characteristic is attained at a large number of dyadic intervals and not only on one chain.

We turn to the case when $I$ is a general interval. First of all, the arguments in case b above are also true for all intervals $I$ due to the $A_2$ characteristic of power weights. On the other hand, if $I$ can be covered by a dyadic interval of a comparable size, then again \e{caseb} holds. Otherwise, let $k$ be such that $I \subset [0, 2^{-(k-1)})$, $I \not\subset [0, 2^{-(k+1)})$ and $|I| \lesssim 2^{-k}$. We distinguish two cases.
\begin{enumerate}
    \item[(i)] One of the following holds: $(1+\alpha)2^{-(k+1)} \in I$, $(1-\alpha) 2^{-k} \in I$, $(1+\alpha)2^{-k} \in I$, $(1-\alpha) 2^{-(k-1)} \in I$. All four cases are similar, so we only consider the second one. For $\sigma$ we have
    \begin{equation}
        \int_{I} \sigma (x) dx \sim 2^{-k (\alpha -1) } |I|.
    \end{equation}
    As for $w$ we write
    \begin{equation}\label{caseiw}
        \int_{I} w(x) dx \sim \big( (1-\alpha)2^{-k} - l(I) \big) 2^{ -k ( 1 - \alpha) } + \int_{(1-\alpha )2^{-k} }^{ r(I) } w(x)dx,
    \end{equation}
    where $l(I)$ and $r(I)$ are the left and right endpoints of $I$
    \item[(i.1)] If $r(I) < (1-\alpha )2^{-k} + \alpha 2^{-(k+1)} $, then we have
    \begin{equation}
         \int_{I} w(x) dx \sim |I| 2^{-k( 1 - \alpha) },
    \end{equation}
    and so
    \begin{equation}
        \frac{1}{|I|^2} \big( \int_I w \big) \cdot \big( \int_I \sigma \big) \lesssim 1.
    \end{equation}
    \item[(i.2)] If $(1-\alpha )2^{-k} + \alpha 2^{-(k+1)} < r(I)$, then using the computation in \e{calculationforw}, we have
    \begin{equation}
        \int_{(1-\alpha )2^{-k} }^{ r(I) } w(x) dx \lesssim 2^{-k(2-\alpha ) }.
    \end{equation}
    Hence, we obtain
    \begin{align*}
        \frac{1}{|I|^2} \big( \int_I w \big) \cdot \big( \int_I \sigma \big) &\lesssim \frac{1}{|I|^2} 2^{- k (2 - \alpha) } \cdot |I|2^{-k(1-\alpha)} \lesssim \frac{2^{-k}}{|I|} \lesssim \frac{1}{\alpha},
    \end{align*}
    where the last step is due to $l(r) < (1-\alpha )2^{-k} < (1-\alpha )2^{-k} + \alpha 2^{-(k+1)} < r(I)$.
    \item[(ii)] Let us have $(1-\alpha) 2^{-k} \notin I$, $(1+\alpha) 2^{-k} \notin I$ and $2^{-k} \in I$. Without loss of generality we can assume $r(I) - 2^{-k} \leq 2^{-k} - l(I)$. Then, we have $|I| \sim ( 2^{-k} - l(I) )$. Furthermore,
    \begin{align*}
         \int_{I} \sigma (x) dx \sim \int\limits_{l(I)}^{2^{-k}} \sigma (x) dx,
    \end{align*}
    and
    \begin{align*}
        \int_{I} w (x) dx \sim \int\limits_{l(I)}^{2^{-k}} w (x) dx.
    \end{align*}
    Thus, as $w$ and $\sigma$ are just power weights on $[ (1-\alpha)2^{-k} , 2^{k} )$, we are in the realm of the usual power weights and the estimate \e{A_2constant} holds.
\end{enumerate}

\subsection{Construction of the sparse family}
Let us take the following sparse family:
\begin{equation}
    \mathcal{S} := \{  [2^{-k} - 2^{-j} , 2^{-k} ) \; : \text{ for all } k,j \in \mathbb{N} \text{ and } j \geq a+k \}.
\end{equation}
We will denote by $B_{k,j} := [2^{-k} - 2^{-j} , 2^{-k} )$.
Using \e{calculationforsigma}, we have
\begin{equation}
    M_{ B_{k,j} } ( \sigma ) \sim 2^{k} \int\limits_0^{ 2^{-k} } \sigma (x) dx \sim \frac{ 2^{ k(1- \alpha ) } }{ \alpha },
\end{equation}
and the corresponding strong-sparse operator is
\begin{equation}\label{counterexamplestrongsparse}
    \ZA_{\ZS}^* f(x) := \sum\limits_{k=1}^{\infty} \frac{ 2^{ k(1 - \alpha ) } }{ \alpha } \sum\limits_{j=a+k}^{\infty} \Z1_{B_{j,k} } (x).
\end{equation}
\subsection{The lower bound}
We claim that
\begin{equation}\label{counterexampletoprove}
    \int\limits_0^1 \ZA_{\ZS}^* (\sigma ) (x)^2 w(x) dx \sim \frac{1}{\alpha ^4} \int\limits_0^1 \sigma (x) dx.
\end{equation}
By \e{counterexamplestrongsparse}, we can write
\begin{equation}\label{sumandsparse}
    \int\limits_0^1 S^* (\sigma ) (x)^2 w(x) dx \sim \sum\limits_{k=1}^{\infty} \frac{ 2^{ 2k(1-\alpha) } }{ \alpha^2 } \int\limits_0^1 \Big( \sum\limits_{j=k+a}^{\infty } \Z1_{ B_{k,j} } (x) \Big)^2 w(x)dx.
\end{equation}
We make a change of variables in the integral and see that it realizes the sharp constant for the regular sparse operator. Putting $y = \frac{x - (1-\alpha )2^{-k} }{2^{-k} \alpha}$, we can write
\begin{align*}
    \int\limits_0^1 \Big( \sum\limits_{j=k+a}^{\infty } \Z1_{ B_{k,j} } (x) \Big)^2 w(x)dx &= \alpha 2^{\alpha k - 2k} \int\limits_0^1 \Big( \sum\limits_{j=1}^{\infty } \Z1_{ [0, 2^{-j} ) } (y) \Big)^2 y^{\alpha - 1} dy \\
    & \sim \alpha 2^{\alpha k - 2k} \cdot \frac{1}{\alpha^3} = \frac{ 2^{\alpha k - 2k} }{\alpha^2},
\end{align*}
where the penultimate estimate follows by the sparse estimate. Plugging this into \e{sumandsparse}, we obtain
\begin{equation*}
    \int\limits_0^1 \ZA_{\ZS}^* (\sigma ) (x)^2 w(x) dx \sim \sum\limits_{k=1}^{\infty} \frac{ 2^{ 2k(1-\alpha) } }{ \alpha^2 } \cdot \frac{ 2^{\alpha k - 2k} }{\alpha^2} \sim \frac{1}{\alpha ^5} \sim \frac{1}{\alpha^4} \int\limits_0^1 \sigma .
\end{equation*}
This finishes the proof of \trm{Strong}.

\section{Proof of \trm{Chain}}\label{ChainCase}
We can assume that the intervals in the sparse family are in some bounded interval, and the general case will follow by a limiting argument. Let us enumerate the intervals of the sparse family $\ZS$.
\begin{equation*}
B_1 \supset B_2 \supset \cdots \supset B_k \supset \cdots.
\end{equation*}
Let $g \in L^2(w)$ and denote by $\pi (B_i) \supset B_i$ the largest interval for which $M_{B_i} (g) \leq 2 \langle g \rangle _{ \pi (B_i ) }$. We can enumerate $\{ \pi (B_i) \}$ by $A_1 \supset A_2 \supset \dots $. As $B_i$'s are embedded, we can choose $A_i$ to be embedded and also sparse.

Consider the following function
\begin{align*}
\tilde{g} (x) = \begin{cases}
\frac{1}{| A_{i} \setminus A_{i+1} |} \int_{ A_{i} \setminus A_{i+1} } g, \quad x\in A_{i} \setminus A_{i+1} \text{ for some }i\in\mathbb{N}, \\
g(x), \text{ otherwise.}
\end{cases}
\end{align*}
First of all, it is clear that
\begin{equation}\label{meanpreserved}
    \int_{A_i} g = \int_{A_i} \tilde{g}.
\end{equation}
By construction, $\langle g \rangle _{A_i } \leq 2 \langle g \rangle_{A_{i+1}}$, so $\frac{1}{| A_i \setminus A_{i+1} |} \int_{A_i \setminus A_{i+1} } g \lesssim \frac{1}{|A_i|} \int_{A_i} g \leq 2 \langle g \rangle_{A_{i+1}}$. Let $B \in \ZS$ be such that $A_i = \pi (B)$. Then, $A_{i+1} \subsetneq B$ due to the choice of $\pi (B)$. Using these observations, we have
\begin{align*}
    \langle \tilde{g} \rangle_{A_i} &= \langle g \rangle_{A_i} = \frac{1}{|A_i| } \Big( \int_{A_{i+1} } g + \int_{A_i \setminus A_{i+1} } g \Big) \\
    & \lesssim \frac{1}{|A_i| } \int_{A_{i+1} } \tilde{g} + \frac{1}{ |B \setminus A_{i+1} | } \int_{B \setminus A_{i+1}} \tilde{g}\\
    & \lesssim \frac{ 1 }{|B|} \int_B \tilde{g}.
\end{align*}
We conclude, that for all $x$
\begin{equation}\label{dominationbysparse}
\ZA_{\ZS}^* g(x) \lesssim \ZA_{\ZS} \tilde{g} (x).
\end{equation}
We turn to the norm of $\tilde{g}$.
\begin{align*}
\int_{\ZR} \tilde{g}^2 w & = \sum\limits_i \int_{ A_i\setminus A_{i+1} } \tilde{g} ^2 w + \int_{\ZR \setminus \cup(A_i\setminus A_{i+1}) } g^2 w \\
& \leq \sum\limits_i \left( \frac{1}{| A_i \setminus A_{i+1} |} \int_{ A_i \setminus A_{i+1} } g \right)^2 w( A_i \setminus A_{i+1} ) + \int_{\ZR \setminus \cup(A_i\setminus A_{i+1}) } g^2 w  \\
& \leq \sum\limits_{i} \frac{w(A_i \setminus A_{i+1}) \cdot \sigma (A_i \setminus A_{i+1})}{|A_i \setminus A_{i+1}|^2} \int_{A_{i+1} \setminus A_{i}} g^2 w + \int_{\ZR \setminus \cup(A_i\setminus A_{i+1}) } g^2 w \\
& \leq \sum\limits_{i} \frac{w(A_i) \cdot \sigma (A_i) }{|A_i|^2} \int_{A_{i+1} \setminus A_{i}} g^2 w + \int_{\ZR \setminus \cup(A_i\setminus A_{i+1}) } g^2 w \\
& \lesssim [w]_{A_2} \int_{\ZR} g^2 w.
\end{align*}
Combining the last estimate, \e{dominationbysparse} and the sparse bound \e{sharpweightedsparse} we conclude
\begin{align*}
    \| \ZA_{\ZS}^* g \|_{L^2(w) } \lesssim \| \ZA_{\ZS} \tilde{g} \|_{L^2(w) } \lesssim [w]_{A_2} \| \tilde{g} \|_{L^2(w)} \lesssim [w]_{A_2}^{\frac{3}{2}} \|g \|_{L^2(w)}.
\end{align*}
And the proof of \trm{Chain} is complete.

\bibliographystyle{amsalpha}

\bibliography{references}


\end{document}